\documentclass[12pt]{article}
\usepackage{geometry}                
\usepackage{graphicx}
\usepackage{amssymb, enumerate}
\usepackage{epstopdf}
\usepackage{amsmath}
\usepackage{amssymb}
\usepackage{amsthm}
\newtheorem{theorem}{Theorem}
\newtheorem{lemma}[theorem]{Lemma}
\newtheorem{corollary}[theorem]{Corollary}
\newtheorem{proposition}[theorem]{Proposition}

\newtheorem{conjecture}[theorem]{Conjecture}

\newtheorem{definition}[theorem]{Definition}

\title{Monochromatic paths  in $2$-edge-coloured  graphs and hypergraphs}
\author{Maya Stein\thanks{Department for Mathematical Engineering and Center for Mathematical Modeling, University of Chile. Support by ANID Regular Grant 1221905, by FAPESP-ANID Investigaci\'on Conjunta grant 2019/13364-7, and by ANID PIA CMM FB210005.}}

\date{}                                           

\begin{document}
\maketitle

\begin{abstract}
We answer a question of Gy\'arf\'as and S\'ark\"ozy from 2013 by showing that every $2$-edge-coloured complete $3$-uniform hypergraph can be partitioned into two monochromatic tight paths of different colours. \\ We also give a lower bound for the number of tight paths needed to partition any $2$-edge-coloured complete $r$-partite $r$-uniform hypergraph. Finally, we show that any $2$-edge-coloured complete bipartite graph has a partition into a monochromatic cycle and a monochromatic path, of different colours, unless the colouring is a split colouring.
\end{abstract}
\section{Introduction}
Monochromatic partitions and covering problems are a part of Ramsey theory in the wider sense, and the topic has seen increased activity in the last decade. As in the classical Ramsey problem, one wishes to find certain monochromatic subgraphs in a graph $G$ whose edges are coloured with two colours. Often this is the complete graph on $n$ vertices, $K_n$. 
Instead of just one monochromatic copy as in Ramsey's theorem, in monochromatic partitioning problems we aim to find a collection of such copies that together cover the whole vertex set of the host graph. 

At the centre of this area 
lies an observation by Gerencs\'er and Gy\'arf\'as~\cite{GG67}, which  
states that in any $2$-colouring of the edges of $K_n$ there are two disjoint monochromatic paths, of different colours, that together cover the vertex set of $K_n$. If we allow each of $\emptyset$, a vertex, an edge to count as a cycle, then the same statement holds  substituting the two paths with two cycles: this was conjectured by  Lehel in the 1970's (see~\cite{Aye99}), and resolved by Bessy and Thomass\'e in 2010~\cite{BT10}, with earlier asymptotic results in~\cite{Allen, LRS98}. 

Generalisations to hypergraphs have been considered, where several notions of cycles and paths exist, and there are 
some results for loose, $\ell$- and Berge-paths and cycles (see e.g.~\cite{BS18, Gy16, GS14, GS13, loose}), 
but for tight paths and cycles, not much is known. A {\it tight path} in an $r$-uniform hypergraph is a ordered sequence of vertices such that each $r$ consecutive vertices form an edge. {\it Tight cycles} are defined analogously on  cyclic sequences of vertices. Usually, sets of size between $1$ and $r-1$ are  allowed as tight paths or cycles, although in this paper, this is not needed, unless $n\le 5$. We will allow a tight path to be empty, i.e., to have no edges or vertices.

In 2013, Gy\'arf\'as and S\'ark\"ozy~\cite{GS13} suggested the following problem: `decide whether  there  is  a  partition  into  a  red  and  a  blue  tight  path  in  every  2-coloring  of  a complete 3-uniform hypergraph'. The question was reiterated in Gy\'arf\'as's survey~\cite{Gy16}, but no advances were made until  2019, when Bustamante, H\`an, and the  author~\cite{BHS16} proved that in any $2$-edge-colouring of the  $3$-uniform complete graph~$\mathcal K^{(3)}_n$,  two tight monochromatic cycles  of different colours partition almost all
 its vertices (all but $o(n)$ vertices, to be precise). As tight cycles contain spanning tight paths, this answers Gy\'arf\'as and S\'ark\"ozy's question asymptotically.
 
   Here, we solve the problem posed by Gy\'arf\'as and S\'ark\"ozy~\cite{GS13} exactly.
 
\begin{theorem}\label{2paths}
For any $n\in\mathbb N$ and  any $2$-colouring of the edges of the 3-uniform complete hypergraph $\mathcal K^{(3)}_n$, there are two monochromatic vertex-disjoint tight paths, of different colours, that together cover all the vertices.
\end{theorem}

It is not possible to improve Theorem~\ref{2paths} by replacing the tight paths with tight   cycles: Lo and Pfenninger~\cite{LoPf20} exhibit colourings of arbitrarily large $\mathcal K^{(r)}_n$, for  $k\ge 3$, which do not admit a partition into two tight cycles of different colours. However,  it has been announced
(see~\cite{LoPf20}) that a partition into  two tight monochromatic cycles, possibly of the same colour, does exist for the case $r=3$.

 Lo and Pfenninger~\cite{LoPf20} also show an analogue of the result from~\cite{BHS16} 
        for  $4$-uniform hypergraphs (i.e., they show that  for any $2$-edge-colouring of the  $4$-uniform complete graph $\mathcal K^{(4)}_n$,  two tight monochromatic cycles  of different colours partition almost all
the vertices), and indicate they believe their result should generalise to other uniformities. We believe that Theorem~\ref{2paths} should generalise to arbitrary uniformities as well:

\begin{conjecture}\label{2pathsconj}
For any $r, n\in\mathbb N$ and  any $2$-colouring of the edges of the $r$-uniform complete hypergraph $\mathcal K^{(r)}_n$, there are two monochromatic vertex-disjoint tight paths, of different colours, that together cover all the vertices.
\end{conjecture}
        
In a variant of the original problem in graphs, the underlying complete graph is replaced with the complete balanced bipartite graph $K_{n,n}$. Pokrovskiy~\cite{P14} showed that these graphs can be partitioned into two monochromatic paths, unless the colouring is a {\it split colouring}, that is, a colouring where each colour induces the disjoint union of two complete bipartite graphs. (It is easy to see that if these complete bipartite graphs are  sufficiently unbalanced, a partition into two monochromatic paths becomes impossible. On the other hand,  a partition into three monochromatic paths is always possible in a split colouring.)
\begin{theorem}[Pokrovskiy~\cite{P14}]\label{Pokbip}
Any $2$-colouring of the edges of $K_{n,n}$ that is not a split colouring allows for a partition of $V(K_{n,n})$ into two monochromatic paths  of different colours.
\end{theorem}
Gy\'arf\'as and Lehel~\cite{Gy83, GyLe73} proved  earlier that a partition of all but one vertex of~$K_{n,n}$ exists.  

We give a shorter proof of Pokrovskiy's result, Theorem~\ref{Pokbip}, and improve it to a partition into a path and a cycle, as follows.

\begin{theorem}\label{bip}
Any $2$-colouring  of the edges of $K_{n,n}$ that is not a split colouring allows for a partition of $V(K_{n,n})$ into a monochromatic path and a monochromatic cycle, of different colours.
\end{theorem}

It is not possible to obtain a partition into two cycles, or to choose the colour of the cycle in Theorem~\ref{bip}. This can be seen by considering a split colouring of $K_{n,n}$ with colours blue and red, and then recolouring in blue one of the red edges. If the sizes of the classes of the split colouring are  chosen sufficiently different, then no partition into two monochromatic cyles exists. Moreover, given any partition into a monochromatic path and a monochromatic cycle, it is not difficult to see that the monochromatic path
  has to use the recoloured edge, and is therefore blue. However, our proof will show that in certain situations it is possible to choose the colour of the cycle (see Corollary~\ref{coro13}).

Finally, we return to   $r$-uniform hypergraphs, now for arbitrary uniformities $r$. Having studied the problem for the graphs  $K_n$ and $K_{n,n}$, and for the hypergraph~$\mathcal K^{(r)}_n$, it seems natural to consider a multipartite
 hypergraph variant as well. We propose to   replace the underlying complete  $r$-uniform  hypergraph with $\mathcal K^{(r)}_{r\times n}$, the complete balanced $r$-partite $r$-uniform   hypergraph with $n$ vertices in each partition class. 
It is not clear whether there is a function $f(r)$ such that for any 2-colouring of the edges of  this hypergraph there is a partition into $f(r)$  monochromatic tight paths. But  if such a function $f(r)$ exists, then $f(r)\ge r+1$, as the following result shows. 

\begin{proposition}\label{example}
For all $n,r\in\mathbb N$ with $n\ge 3^{r+2}$ and $r\ge 2$, one can colour the edges of~$\mathcal K^{(r)}_{r\times n}$ with two colours so that  $V(\mathcal K^{(r)}_{r\times n})$ cannot be partitioned into less than $r+1$ monochromatic tight paths.
\end{proposition}

Observe that for $r=2$, the bound from Proposition~\ref{example} coincides with the  bound coming from split colourings, which is 3. To prove the proposition, we introduce a generalisation of the split colouring to hypergraphs.


The rest of the paper is organised as follows. Section~\ref{tight} contains the proof of Theorem~\ref{2paths}, showing the slightly stronger Proposition~\ref{bipath}. Section~3 is dedicated to $r$-partite $r$-uniform hypergraphs and split colourings, and Proposition~\ref{example} is proved. In Section~4, we give the brief proof of Theorem~\ref{bip}, based on 
Theorem~\ref{Pokbip};
 then offer an alternative shorter proof of Theorem~\ref{Pokbip}; and finally present a quick application of Theorem~\ref{bip} to $3$-edge-colourings of complete and complete bipartite graphs.

\section{Two tight paths for $\mathcal K^{(3)}_n$}\label{tight}

In this section we show Theorem~\ref{2paths}, thus answering the question of Gy\'arf\'as and S\'ark\"ozy from~\cite{GS13}. Similarly as in the proof for graphs from~\cite{GG67}, the starting point for our proof is  a path, in our case a tight path, that switches colours at most once and has maximal length with this property. Let us define this object in precise terms.

\begin{definition}[Tight bicoloured path]\label{bipa}
Given  any $2$-colouring of the edges of the 3-uniform complete hypergraph~$\mathcal K^{(3)}_n$,  we 
call a tight path $P=v_1v_2\ldots v_k$ {\em bicoloured} if there is an index $\ell\in\{1,\ldots,k-1\}$ such that each edge $v_{i-1}v_{i}v_{i+1}$ with $2\le i\le\ell$ has the same colour, and all other edges of $P$ have the other colour. Call $v_\ell$ the {\em turning point} of $P$. 
\end{definition}
Note that  monochromatic paths are allowed in Definition~\ref{bipa}: namely, a bicoloured path with turning point $v_\ell$ is monochromatic if and only if $\ell=1$ or  $\ell=k-1$. As another example, if we have two differently coloured edges $v_1v_2v_3$ and $v_2v_3v_4$, then they form a bicoloured path $v_1v_2v_3v_4$ with turning point $v_2$.

Clearly, every bicoloured tight path $P$ contains two disjoint monochromatic paths $P_1$, $P_2$ of distinct colours (and moreover, if $P$ has at least $6$ vertices, then each of $P_1$, $P_2$ is either empty or has at least an edge). So in order to prove Theorem~\ref{2paths}, it suffices to show the following proposition.

\begin{proposition}\label{bipath}
For any $n\in\mathbb N$ and  any colouring of the edges of the 3-uniform complete hypergraph $\mathcal K^{(3)}_n$, there is a spanning bicolored tight path.
\end{proposition}

Note that Proposition~\ref{bipath} is slightly stronger than Theorem~\ref{2paths} in the sense that two monochromatic tight paths of distinct colours can not always be concatenated to a bicoloured path.

\begin{proof}[Proof of Proposition~\ref{bipath}]

We take  a longest bicoloured tight path $P$ in $\mathcal K^{(3)}_n$. Say $P=v_1v_2\ldots v_k$, with $v_1v_2v_3$ red, and with  turning point $v_{\ell}$ (for some $\ell\in\{1,\ldots,k-1\}$). 

If $k=n$, we are done, so we assume  $k<n$. Then there is a vertex $w$ which is not on~$P$. Because of the maximality of $P$, we know that $w$ cannot be added to $P$ via the edge $wv_1v_2$ or via the edge $v_{k-1}v_kw$. Therefore, $P$ is not monochromatic and thus $$1<\ell<k-1.$$ 

Note that reversing $P$  makes $v_{\ell +1}$ the turning point of the reversed bicoloured tight path.
This means that  we can  invoke colour symmetry to be able to assume that 
\begin{equation}\label{2}
\text{the edge $v_\ell v_{\ell+1} w$ is red. }
\end{equation}

We claim that 
\begin{equation}\label{3}
\text{$v_{\ell+1} w v_{\ell+2}$ is blue.}
\end{equation}

In order to see~\eqref{3} by contradiction, assume  that
 $v_{\ell+1} w v_{\ell+2}$ is red.
Consider the tight path $P':=v_1v_2\ldots v_\ell v_{\ell+1} w v_{\ell+2} 
\ldots  v_{k-1}v_k$. If  $\ell=k-2$, then by~\eqref{2},  $P'$ is monochromatic (and thus bicoloured). If $\ell<k-2$, then, because of~\eqref{2}, no matter which colour the edge $w v_{\ell+2} v_{\ell+3}$ has,  $P'$ is a  bicoloured tight path. It is also longer than $P$. This is a contradiction, as $P$ was chosen as a maximum length bicoloured tight path. Hence we proved~\eqref{3}.

Next, we claim that 
\begin{equation}\label{4}
\text{$v_1wv_{\ell +1}$ is red.}
\end{equation}

 Indeed, in order to see~\eqref{4}, it suffices to consider the tight path $$P'':=v_\ell v_{\ell-1}\ldots v_2v_1wv_{\ell+1}v_{\ell+2}\ldots v_k.$$ If the edge $v_1wv_{\ell +1}$ was blue, then by~\eqref{3}, and no matter which colour the edge $v_2v_1w$ has,  $P''$ is a bicoloured tight path on $V(P)\cup\{w\}$. This contradicts the maximality of $P$. Thus we proved~\eqref{4}.
 
 Similarly, by considering the tight path $v_1 v_{2}\ldots v_{\ell}v_{\ell+1}wv_kv_{k-1}\ldots v_{\ell+2}$, which cannot be bicoloured, we can use~\eqref{2} to see that 
 \begin{equation}\label{5}
\text{the edge $v_{\ell+1}wv_k$ is blue.}
\end{equation}

Now, consider the edge $v_1 w v_k$. If this edge is red, then because of~\eqref{2} and~\eqref{4}, the tight path
$$v_2v_3\ldots v_\ell v_{\ell+1}wv_1v_kv_{k-1}v_{k-2}\ldots v_{\ell +2}$$ is bicoloured (no matter which colour the edge $v_1v_kv_{k-1}$ is). Similarly, if $v_1 w v_k$ is blue, then by~\eqref{5} and~\eqref{3},  
$$v_{\ell}v_{\ell-1}\ldots v_{2}v_{1}v_kwv_{\ell+1}v_{\ell +2}\ldots v_{k-2}v_{k-1}$$
  is a bicoloured tight path no matter the colour of $v_2v_1v_k$. In either case, we obtain a contradiction to the maximality of the bicoloured tight path~$P$, which concludes the proof.
\end{proof}

\section{Split colourings}
\subsection{Split colourings in hypergraphs}
In this section, we consider $2$-colourings of the edges of the $r$-partite $r$-uniform hypergraph with $n$ vertices in each partition class, denoted by~$\mathcal K^{(r)}_{r\times n}$. 
For graphs, the following definition coincides with the usual definition of split colourings.

\begin{definition}[Hypergraph split colourings]
Let $V_1, V_2, \ldots, V_r$ be the partition classes of $\mathcal K^{(r)}_{r\times n}$. For each $i$, partition $V_i$ into two non-empty  sets $V_i^1$, $V_i^2$. 
Colour  each edge having an even number of vertices in $\bigcup_{1\le i\le r} V_i^{1}$ in red, and colour all other edges blue. Any colouring obtained in this way is called a {\em split colouring}.
\end{definition}

We now give the proof of Proposition~\ref{example}.
\begin{proof}[Proof of Proposition~\ref{example}]
Consider a split colouring of $\mathcal K^{(r)}_{r\times n}$ with classes 
$V^1_1, V_1^2, V^1_2$, $V^2_2, \ldots, V^1_r, V^2_r$, where for $1\le i\le r$,  the class $V^1_i$ has size $3^{i}$, and the class $V^2_i$ has size $n-3^{i}$.
Let $\mathcal P$ be a set of disjoint monochromatic tight paths that cover all the vertices. 

Observe that because of  the structure of the split colouring, each  path in $\mathcal P$ can only meet at most one of $V^1_i$, $V^2_i$, for each $1\le i\le r$, and so,
\begin{equation}\label{dovescry}
\text{for each $P\in\mathcal P$ and each $i\in [r]$, one of $V^1_i\cap V(P)$, $V^2_i\cap V(P)$ is empty.}
\end{equation}
Our aim is to show that $|\mathcal P|\ge r+1$.
For this, we let  $\mathcal P_i$ denote the set of all paths in $\mathcal P$ that meet $V^1_1\cup\ldots\cup V^1_i$ (in particular $\mathcal P_0=\emptyset$), set $V_0:=\emptyset$, and use induction on $i$ to show that for $0\le i\le\ell\le r$,
\begin{equation}\label{rollerskatejams}
\text{at most $\textstyle\sum_{j=1}^i|V_j^1|+i$ vertices in $V_\ell$ are covered by paths from $\mathcal P_i$.}
\end{equation}
This is trivially true for $i=0$ and all $0\le \ell\le r$. For $i\ge 1$, because of~\eqref{dovescry}, the paths in $\mathcal P_i\setminus \mathcal P_{i-1}$ cannot meet $V_i^2$, and therefore,  by induction, 
 the number of vertices in $V_\ell$ covered by paths from $\mathcal P_i$ 
is at most
$|V_i^1|+1+\sum_{j=1}^{i-1}|V_j^1|+i-1=\sum_{j=1}^i|V_j^1|+i$. This proves~\eqref{rollerskatejams}.

Now, observe that for all $1\le i\le r$, $$|V_i^1|= 3^{i}>\sum_{j=1}^{i-1}3^j+i-1=\sum_{j=1}^{i-1}|V_j^1|+i-1.$$ So, by~\eqref{rollerskatejams} for $\ell=i$,  we know that $\mathcal P_i\setminus \mathcal P_{i-1}\neq\emptyset$ for all $1\le i\le r$. In particular, $|\mathcal P_r|\ge r$.

Moreover, by~\eqref{rollerskatejams} for $\ell =i=r$,  the paths from $\mathcal P_r$
cover at most  $$\sum_{j=1}^{r}|V_j^1|+r= \sum_{j=1}^{r}3^j+r\le 3^{r+2}-3^r-1\le n-3^r -1= |V_r^2|-1$$ vertices  in $V_r$, where we used the fact that by assumption,  $n\ge 3^{r+2}$. As~$V_r^2\subseteq V_r$ has to be covered by paths from~$\mathcal P$, we deduce that $|\mathcal P|\ge |\mathcal P_r|+1\ge  r+1$, which is as desired.
\end{proof}

\subsection{Split colourings in graphs}\label{sec3.2}
For graphs, it is known that split  colourings can be characterised by the existence of certain substructures. Pokrovskiy~\cite{P14} showed that
 a $2$-colouring of $K_{n,n}$ is split if and only if all vertices see both colours and there is no spanning monochromatic component. 
 In Lemma~\ref{charac} below, we give a new characterization of colourings that are not split colourings or one of their close relatives: monochromatic colourings or  {\it V-colourings}. The latter class is defined now:
 
 \begin{definition}[V-colouring]\label{Vcoldef} Call a colouring of the edges of $K_{n,n}$ a {\it V-colouring} if it has the property that each of the colours spans a complete bipartite graph.
  \end{definition}
Note that the set of all vertices seeing both colours coincides with one of the partition classes of $K_{n,n}$, giving the colouring the form of the letter V.
 
For the characterisation in  Lemma~\ref{charac}, the following notation  will be convenient. 

\begin{definition}[Bicoloured cycle, good cycle]\label{goodcycledef}
In a $2$-edge-coloured $K_{n,n}$, call a cycle $C=v_1v_2\ldots v_kv_1$ 
 {\em bicoloured} if for some $\ell\in\{2,\ldots, k\}$  each edge $v_{i}v_{i+1}$ with $1\le i\le\ell-1$ has one colour, and all other edges have the other colour. Call
 $v_1$ and $v_\ell$ the {\em turning points of $C$}. If  $v_1$ and $v_\ell$ lie in distinct partition classes, call $C$ {\em good}. 
 \end{definition}
 Note that, in contrast with Definition~\ref{bipa},  monochromatic cycles 
  are  {not} bicoloured.

\begin{lemma}\label{charac}
Given a colouring $c$ of the edges of $K_{n,n}$, the following are equivalent:
\begin{enumerate}[(a)]
\item $c$ is either  a split colouring, or a V-colouring, or monochromatic;
\item there are no  good cycles for $c$;
\item there is no  good $C_4$ for $c$.
\end{enumerate}
\end{lemma}
\begin{proof}
We only show that $(c)$ implies $(a)$, as the implications $(a)\Rightarrow (b)$ and $(b)\Rightarrow (c)$ are straightforward. 
 Unless $c$ is mono\-chromatic, there is a vertex $v$ such that both sets $X^r:=\{w\in N(v):\text{$vw$ is red}\}$ and $X^b:=\{w\in N(v):\text{$vw$ is blue}\}=N(v)\setminus X^r$ are non-empty. 
  Because of~$(c)$, each vertex  from the  partition class that contains $v$  is  monochromatic to each of $X^r$,  $X^b$, but not to $X^r\cup X^b$. So $c$ is a split colouring or a V-colouring.
\end{proof}

\section{Bipartite graphs}
In Section~\ref{sec4.1} we prove   Theorem~\ref{bip}, and in Section~\ref{sec4.2} we give a  new   proof of  Theorem~\ref{Pokbip}. We will need Definition~\ref{goodcycledef} and Lemma~\ref{charac} from Section~\ref{sec3.2}. In Section~\ref{sec4.3}, we show an application of Theorem~\ref{bip} to $3$-edge-colourings.

\subsection{Partition into a  path and a cycle}\label{sec4.1}
We start by observing that clearly, in any $2$-edge-coloured  $K_{n,n}$, 
every bicoloured cycle~$C$ 
 contains two disjoint mono\-chromatic paths of distinct colours spanning the same set of vertices as $C$. Also, any two disjoint 
monochromatic paths  of distinct colours   can be transformed into a bicoloured or monochromatic cycle, as follows.
\begin{lemma}\label{convert}
Let $P_1$ and $P_2$ be disjoint monochromatic paths of distinct colours in an $2$-edge-coloured  $K_{n,n}$ such that $V(K_{n,n})=V(P_1)\cup V(P_2)$. Then there is  a spanning  cycle that is either bicoloured or monochromatic.
\end{lemma}
\begin{proof}
We can assume that $n\ge 2$ and that $P_1=v_1\ldots v_\ell$ has at least two vertices. If $P_2$ is empty, $v_1$ and $v_\ell$ are in opposite partition classes, and we can add the edge $v_\ell v_1$ to $P_1$ to obtain the desired cycle. If $P_2$ has only one vertex $w$, the edges $v_\ell w$ and $wv_1$ exist, and can be added to $P_1$  with the same outcome.

So assume $P_2=w_1\ldots w_k$ with $k\ge 2$.
Observe that one of the endpoints of $P_1$, say $v_1$, does not lie in the same partition class as $w_1$ (otherwise the three vertices $v_1, v_\ell, w_1$ are in the same partition class which then would be larger than the other partition class, but both have size $n$). Similarly, also $v_\ell$, $w_k$ are in opposite partition classes and, therefore, we can transform the  paths $P_1$, $P_2$ into a bicoloured cycle by adding the edges $v_1w_1$ and~$v_\ell w_k$.
\end{proof}

From the observation above and Lemma~\ref{convert}, it follows that Theorem~\ref{Pokbip} is equivalent to the following.
\begin{lemma}\label{bipi}
Every $2$-edge-coloured  $K_{n,n}$ has a spanning  bicoloured or monochromatic  cycle, unless the colouring is split.
\end{lemma}

Now Theorem~\ref{bip} follows immediately from Lemma~\ref{bipi} and the next lemma.
\begin{lemma}\label{bicospan}
If a $2$-edge-coloured  $K_{n,n}$ has a spanning bicoloured cycle, then it has a partition into a monochromatic cycle and a monochromatic path, of distinct colours.
\end{lemma}
\begin{proof}
 Among all spanning bicoloured  cycles, we choose a cycle $C$ having the maximum number of red edges. Say
$C=v_1v_2\ldots v_kv_1$, with $v_1v_2$ red, and with  turning points $v_1, v_\ell$. 
If $C$ is good, we are done, since the edge  $v_1v_\ell$ generates a monochromatic cycle with one of the two paths $v_2v_3\ldots v_{\ell-1}$ or $v_{\ell +1}v_{\ell +2}\ldots v_k$, while the other path is  monochromatic  in the other colour. So assume $C$ is not good, that is, vertices $v_1, v_\ell$ are in the same partition class. In particular, $v_1v_{\ell +1}$ and $v_\ell v_k$ are edges.

If the edge $v_1v_{\ell +1}$ is red, then $k>\ell+1$ (as $v_kv_1$ is blue) and the bicoloured cycle 
$(C-v_kv_1-v_\ell v_{\ell+1})\cup v_1v_{\ell +1}\cup v_\ell v_k$
has more red edges than $C$. So $v_1v_{\ell +1}$ is blue, and $K_{n,n}$ decomposes into the red path $v_2\ldots v_\ell$ and the blue cycle
on the remaining vertices (the blue cycle may be just an edge, which is allowed).
\end{proof}

Observe that one can easily modify the last lines of  the proof of Lemma~\ref{bicospan}, arguing that we either have a red path and a blue cycle or both edges $v_1v_{\ell +1}$, $v_\ell v_k$ are blue, giving a spanning bicoloured cycle that is not good and has more red edges than $C$. So, if at the beginning  $C$ is chosen only among the bicoloured spanning  cycles that are not good (if such cycles exist), then we can choose the colour of the cycle.
\begin{corollary}\label{coro13}
Let the edges of $K_{n,n}$ be coloured with red and blue. If there is a spanning bicoloured cycle that is not good, then there is  a partition of $K_{n,n}$ into a red path and a blue cycle.
\end{corollary}

\subsection{Alternative proof of Theorem~\ref{Pokbip}}\label{sec4.2}
We give an alternative proof of Theorem~\ref{Pokbip} by proving Lemma~\ref{bipi}.
We start by excluding a certain structure  outside any longest good cycle. 
Call a $C_4$  {\it balanced} if it has exactly two edges of each colour.

\begin{lemma}\label{lem:goodnotsplit}
If
$c$ is a $2$-edge-colouring of   $E(K_{n,n})$ 
and  $C=v_1\ldots v_kv_1$ is a longest good  cycle for   $c$, then the  restriction $c'$ of $c$ to $K_{n,n}-V(C)$ has no balanced~$C_4$.
\end{lemma}
\begin{proof}
Say $C$ has turning points $v_1\in V_1$, $v_\ell \in V_2$, and $v_1v_2$ is red. For contradiction, assume $c'$ has a balanced $C_4$. Then, after exchanging the roles of $V_1$ and $V_2$ if needed,  there are $x_i, y_i\in V_i\setminus V(C)$, for $i=1,2$, with $x_1x_2$  red, $y_1x_2$ blue, and exactly one of $x_1y_2$, $y_1y_2$ red. We will split our proof accordingly. But first note that by colour symmetry we can assume $v_1x_2$ is  red. As the cycle $(C-v_1v_k)\cup v_1x_2x_1v_k$ is longer than $C$, it is not good. So  $x_1v_k$ is red, and similarly,  $x_1v_2$ is blue.

{\it Case 1: $y_1y_2$ is  red and $x_1y_2$ is blue.}  
If $v_1y_2$ is red, then as above,  also $y_1v_k$ is red. Moreover,  $\ell <k$ (replace $v_kv_1=v_\ell v_1$ with $v_kx_1y_2v_1$ otherwise).
 Replace $v_{k-1}v_kv_1$ with  $v_{k-1}y_2y_1v_kx_1x_2v_1$ to see that $y_2v_{k-1}$ is blue. The good cycle $(C-v_k)\cup v_{k-1}y_2x_1x_2v_1$ is  longer  than $C$, a contradiction. 
So  $v_1y_2$ is blue. Recall that also~$x_1v_2$ is blue. If $y_1v_\ell$ is red then so is $y_2v_{\ell+1}$ (in particular $\ell\neq k$), and the good cycle $(C-v_\ell v_{\ell+1}-v_kv_1)\cup v_\ell y_1y_2v_{\ell+1}\cup v_kx_1x_2v_1$ is longer than $C$, a contradiction. So $y_1v_\ell$ is blue. By maximality of $C$ also $x_2v_{\ell-1}$ is blue (and  thus $\ell\neq 2$), and $(C-v_1v_2- v_{\ell-1}v_\ell)\cup  v_1y_2x_1v_2\cup v_{\ell-1}x_2y_1v_\ell$ is a longer good cycle, a contradiction.

{\it Case 2: $x_1y_2$ is red and  $y_1y_2$ is blue.}  If $y_1v_\ell$ is blue, then by maximality of $C$, so is $y_2v_{\ell -1}$, and $(C-v_kv_1-v_{\ell -1}v_\ell)\cup v_kx_1x_2v_1\cup v_{\ell}y_1y_2v_{\ell -1}$ is a longer good cycle, so $y_1v_\ell$ is red. Note that $\ell\neq k$, as otherwise $(C-v_\ell v_1\cup v_\ell y_1 x_2 v_1)$ is  longer than $C$. Also, $v_{\ell +1}y_2$ is red, as otherwise  $(C-v_kv_1-v_\ell v_{\ell +1})\cup v_kx_1x_2v_1\cup v_{\ell}y_1y_2v_{\ell +1}$ is a longer good cycle. 
Now, $v_ky_1$ is blue, as otherwise 
$(C-v_kv_1-v_\ell v_{\ell+1})\cup v_ky_1v_\ell \cup v_1x_2x_1y_2v_{\ell +1}$ is a longer good cycle.
Consider $(C-v_kv_1-v_\ell)\cup v_kx_1x_2v_1\cup v_{\ell -1}y_2 v_{\ell +1}$ to see that $v_{\ell -1}y_2$ is blue. If
 $x_1v_\ell$ is blue, add $v_2x_1v_\ell$ 
and $v_{\ell-1}y_2y_1v_k$ 
 to $C-v_1-v_{\ell-1}v_{\ell}$ for a longer good cycle, and if $x_1v_\ell$ is red, 
replace $v_kv_1$ and $v_\ell v_{\ell +1}$ with  $v_ky_1x_2v_1$ and $v_\ell x_1y_2v_{\ell +1}$ for a longer good cycle, a contradiction.
\end{proof}

Now we determine which colourings avoid balanced $C_4$'s. 
Call an edge-coloring $c$ of $K_{n,n}$ {\it near-monochromatic}  if  there is a colour that  appears on at most one edge.

\begin{lemma}\label{monopath}
Any edge-colouring $c$ of $K_{n,n}$  having no balanced $C_4$ is near-mono\-chromatic.
\end{lemma}
\begin{proof}
Assume $c$ is not monochromatic. Then some vertex $v$ sees both colours. Let $V_1, V_2$ be the partition classes of $K_{n,n}$, and use symmetry   to assume that $v\in V_1$, and that  $|R_v|\ge |B_v|\ge 1$, where $R_v:=\{u\in V_2: uv\text{ is red}\}$ and $B_v:=V_2\setminus R_v$.

If $|R_v|=1$, then $|B_v|=1$, and since $c$ has no balanced $C_4$, it is near-monochromatic. So assume $|R_v|\ge 2$. Then, repeatedly using the fact that there is no balanced $C_4$, we see that each $x\in V_1\setminus\{v\}$ sends at least one red edge to $R_v$, and thus only red edges to~$B_v$, and therefore, only red edges to $R_v$.
In particular $|B_v|=1$ (otherwise $v$, $x$ and any two vertices from $B_v$ span a balanced $C_4$). Thus  $c$ has exactly one blue edge (from $v$ to $B_v$), and is therefore near-monochromatic.
\end{proof}

We are ready to prove Lemma~\ref{bipi}.

\begin{proof}[Proof of Lemma~\ref{bipi}]
We can assume the colouring $c$ is neither a V-colouring nor  monochromatic, as otherwise
it is easy to find  a spanning  bicoloured or monochromatic  cycle.  By  Lemma~\ref{charac}, $c$ has a  good cycle, let $C_1$ be a largest such cycle.

For any good cycle $C$ in $c$, let $c(C)$ denote the restriction  of $c$ to $K_{n,n}-V(C)$. 
By 
Lemmas~\ref{lem:goodnotsplit} and~\ref{monopath},  
$c(C_1)$ is near-monochromatic, and thus has a monochromatic spanning path, say in red.
Among all good cycles  $C$ with the property that $c(C)$ has a red spanning path, choose  $C_2$, maximising the number of blue edges. Say~$C_2=v_1\ldots v_kv_1$ with $v_1v_2$ red and turning points $v_1$ and $v_\ell$,
and let
 $P=x_1\ldots x_h$ be a red spanning path in $c(C_2)$. As $v_1$ and $v_\ell$ are in distinct partition classes,  one of them, say $v_1$, is in the same partition class as $x_1$.  If both $x_1v_\ell$,  $v_1x_h$ are  blue, the cycle $v_1x_h\cup P\cup x_1v_\ell v_{\ell +1}\ldots v_kv_1$ contradicts our choice of $C_2$.  So at least one of the two  edges, say $x_1v_\ell$, is red. 
Then $(C_2-v_\ell v_{\ell +1})\cup v_\ell x_1Px_hv_{\ell +1}$ is a spanning bicolored cycle.
\end{proof}

\subsection{Three colours}\label{sec4.3}

We now  turn to colourings of complete and complete bipartite graphs with three colours. Then, it is known that the edges of $K_n$ cannot always be partitioned into three paths of {\it distinct} colours, so one drops this condition. In any  $3$-edge-colouring of $K_n$ there is a partition into three monochromatic paths~\cite{P14} and  a partition of almost all vertices into three monochromatic cycles~\cite{3col}. Pokrovskiy~\cite{P14} found $3$-colourings that do not admit a   partition of all vertices into three monochromatic cycles, however, he conjectures~\cite{sparse} that three cycles and one vertex are always sufficient.  

In  $3$-edge-colourings of the complete bipartite graph $K_{n,n}$,  almost all vertices can be partitioned into five monochromatic cycles~\cite{LSS17}, while folklore 3-colour split colourings\footnote{A 3-colour split colouring is obtained from a proper edge colouring of $K_{3,3}$ by replacing each vertex with a set of vertices and each edge with a complete bipartite graph in the colour of the edge.} show that a partition into fewer than five monochromatic cycles or paths is not always possible. It is conjectured that five is indeed the optimal number for partitions into monochromatic  paths~\cite{P14}. 
For more details and variants with more colours, see~\cite{Gy16}.

In~\cite{P14, sparse}, Pokrovskiy shows how the following two lemmas for $2$-edge-colourings can be used to obtain partitions of $3$-edge-coloured graphs.

\begin{lemma}$\!\!${\rm\bf\cite{BKS12, P14}}\label{lem1}
For any red-blue colouring of the edges of $K_{n}$  there is a partition of the vertices into
a red path and a blue balanced complete bipartite graph.
\end{lemma}

\begin{lemma}$\!\!${\rm\bf\cite{sparse}}\label{lem2}
For any red-blue colouring of the edges of $K_{n,n}$  there is a partition of the vertices into
a red path and two blue balanced complete bipartite graphs.
\end{lemma}

We can use Pokrovskiy's approach from~\cite{sparse} together with our Theorem~\ref{bip} to obtain partitions of a $3$-edge-coloured $K_n$ or $K_{n,n}$ into a mix of monochromatic cycles and paths.
For $K_n$, we start with any $3$-colouring of its edges, say in colours red, blue and green. We apply Lemma~\ref{lem1} with blue and green merged, which gives a red path and a balanced complete bipartite graph  that only uses blue and green. 
To this graph we apply Theorem~\ref{bip}  to find a partition into a monochromatic path and a monochromatic cycle, unless the graph is split-coloured, in which case we easily find a partition into three monochromatic cycles.
In total, we obtain a partition of $K_n$ into either {two monochromatic paths and one monochromatic cycle}, or {one monochromatic path and three monochromatic cycles.}

Applying Lemma~\ref{lem2} and Theorem~\ref{bip}  in the same way
 we obtain a partition of~$K_{n,n}$ into either three monochromatic paths and two monochromatic cycles, or into two monochromatic paths and four monochromatic cycles, or into one monochromatic path and six monochromatic cycles. 
 We can reduce the number of elements of the partition by analysing the edges between the two balanced blue-green complete bipartite graphs given by Lemma~\ref{lem2}: 
\begin{lemma}
For any  $3$-edge-colouring of $K_{n,n}$ there is a partition of the vertices into
either three monochromatic paths and two monochromatic cycles or into two monochromatic paths and four monochromatic cycles.
\end{lemma}
\begin{proof}
Say colours are red, blue, green. Lemma~\ref{lem2}  gives a red path and four sets  $A_1, B_1\subseteq V_1$ and $A_2, B_2\subseteq V_2$ (where $V_1, V_2$ are the partition classes of $K_{n,n}$), with $|A_1|=|A_2|\le |B_1|=|B_2|$, and no red edge between $A_1$ and $A_2$ or between $B_1$ and~$B_2$. We can assume the colourings in green and blue on $A_1, A_2$ and $B_1, B_2$ are both split, as by Theorem~\ref{bip}, any non-split colouring allows for a partition into   a monochromatic cycle and a monochromatic path, while a split colouring allows for a partition into three monochromatic cycles, giving a total outcome  which is as desired for the lemma.

First assume  there is a non-red edge $e$ between  $A_j$ and $B_{3-j}$ for $j=1$ or $j=2$. Say $e$ is blue. Then we can find a blue path through $e$ such that the remainder of each of $A_1\cup A_2$ and $B_1\cup B_2$ is a balanced complete  bipartite  graph with a V-colouring and can thus be partitioned into two monochromatic cycles. This gives a total of two paths and four cycles, which is as desired. 

On the other hand, if all edges between $A_1$, $B_2$, and  between $B_1$, $A_2$ are red, then there are two disjoint red cycles covering all of $A_1\cup A_2$ and part of $B_1\cup B_2$. Now either the remainder $B$ of $B_1\cup B_2$ is split, or we can apply Theorem~\ref{bip}. In both cases $B$ partitions into at most  two monochromatic cycles and a monochromatic path. Together with the red path and the two red cycles, we obtain the desired  partition.
\end{proof}

\end{document}